\documentclass[12pt]{article}
\usepackage[a4paper,margin=3cm]{geometry}
\usepackage{amsmath}
\usepackage{amssymb}

\newtheorem{theorem}{Theorem}[section]
\newtheorem{lemma}[theorem]{Lemma}
\newtheorem{corollary}[theorem]{Corollary}
\newtheorem{proposition}[theorem]{Proposition}

\newenvironment{proof}{{\bf Proof.}}{\hfill$\Box$\\}
\newenvironment{proof of}{{\bf Proof of}}{\hfill$\Box$\\}

\newcommand{\R}{\mathbb{R}}

\newcommand{\FF}{\mathcal{F}}

\newcommand{\nn}{\nabla^\mathcal{F}}

\title{{\bf Webs and projective structures on a plane}}
\author{
Wojciech Kry\'nski\thanks{
{\bf Department of Applied Mathematics and Theoretical Physics, Cambridge University, Wilberforce Road, Cambridge, CB3 OWA, UK}\newline and\newline {\bf Institute of Mathematics, Polish Academy of Sciences, ul.~\'Sniadeckich 8, 00-956 Warszawa, Poland}\newline 
E-mail: krynski@impan.pl.}
}
\date{}

\begin{document}
\maketitle
\begin{abstract}
We prove that there is a correspondence between projective structures defined by torsion-free connections with skew-symmetric Ricci tensor and Veronese webs on a plane. The correspondence is used to characterise the projective structures in terms of second order ODEs. 
\end{abstract}

\section{Introduction}
A web is a family of foliations on a manifold. In the present paper we concentrate on the simplest example which is a 3-web on a plane, i.e. a triple of one-dimensional foliations in the general position on $\R^2$ (see \cite{AG,N}). We show that a 3-web defines a projective structure on a plane. The projective structures obtained in this way are very special. Namely, they are defined by linear connections with skew-symmetric Ricci tensor. Additionally, the associated twistor space fibers over the projective space $\R P^1$.

The existence of the fibration in the twistor picture suggests that the projective structures defined by 3-webs are two-dimensional counterparts of so-called hyper-CR Einstein-Weyl structures on $\R^3$ \cite{Dun}. In \cite{DK} we showed that the hyper-CR Einstein-Weyl structures are in a one-to-one correspondence with Veronese webs, i.e. special 1-parameter families of foliations introduced by Gelfand and Zakharevich \cite{GZ} in connection to bi-Hamiltonian systems. The similar phenomenon takes place on the plane. Indeed, one can easily extend a 3-web to a Veronese web on a plane and it is an intermediate step in the construction of a projective structure out of a 3-web. The approach gives new and simple proof of Wong's theorem \cite{W} in the stronger version of Derdzinski \cite{Der}. We also provide local forms of connections with constant skew-symmetric Ricci tensor.

The projective structures defined by connections with skew-symmetric Ricci tensors were investigated recently in \cite{Der,DW,R}. In particular \cite{R} provides a characterisation of this class of projective structures in terms of the associated second order differential equation. We describe an alternative approach in terms of the dual equation at the end of the paper. The result is based on our earlier characterisation of Veronese webs (and more generally Kronecker webs) in terms of ODEs \cite{K2} and involve so-called time-preserving contact transformations \cite{JK}.

\section{3-webs and Veronese webs}
A 3-web on a plane is a triple $\{\FF_1,\FF_2,\FF_3\}$ of one-dimensional foliations such that at any point $x\in \R^2$ any two of them intersect transversely. One can always find a coordinate system $(x,y)$ such that
$$
T\FF_1=\ker dx,\qquad T\FF_3=\ker dy.
$$
Then
$$
T\FF_2=\ker dw=\ker(w_xdx+w_ydy)
$$
for some function $w=w(x,y)$. By the assumption on the transversality of the foliations we get that both $w_x$ and $w_y$ are nowhere vanishing.

Let us notice that any 3-web can be extended to a 1-parameter family $\{\FF_{(s:t)}\}_{(s:t)\in\R P^1}$ of foliations parametrised by points in a projective line $\R P^1$ and such that
\begin{equation}\label{eq1}
\FF_{(1:0)}=\FF_1,\qquad\FF_{(0:1)}=\FF_3,\qquad\FF_{(s_0:t_0)}=\FF_2
\end{equation}
for some fixed point $(s_0:t_0)\in\R P^1$. Namely, we can consider the following family of one-forms
\begin{equation}\label{eq2}
\omega_{(s:t)}=st_0w_xdx+ts_0w_ydy
\end{equation}
and then one sees that condition \eqref{eq1} is satisfied if we define $\FF_{(s:t)}$ by
$$
T\FF_{(s:t)}=\ker\omega_{(s:t)}
$$

The so-obtained family $\{\FF_{(s:t)}\}$ is very special. It depends linearly on a projective parameter $(s:t)$ and it is an example of so-called Veronese webs \cite{GZ,Z}. In general, a 1-parameter family of corank-one foliations on a manifold $M$ of dimension $n+1$ is a Veronese web if any $x\in M$ has a neighbourhood $U$ such that there exist point-wise independent one-forms $\omega_0,\ldots,\omega_n$ on $U$ such that
$$
T\FF_{(s:t)}|_U=\ker{s^n\omega_0+s^{n-1}t\omega_1+\cdots+t^n\omega_n}.
$$
At any point $x\in M$ the mapping
$$
(s:t)\mapsto\R(s^n\omega_0(x)+s^{n-1}t\omega_1(x)+\cdots+t^n\omega_n(x))\in P(T^*_xM)
$$
is a Veronese embeding and it justifies the terminology.

Specifying to $n=2$ we get the following correspondence
\begin{proposition}\label{prop1}
Let $(s_0:t_0)\in \R P^1$ be fixed. Any 3-web on a plane extends uniquely to a Veronese web on $\R^2$ such that \eqref{eq1} is satisfied. Conversely, for a Veronese web $\{\FF_{(s:t)}\}$, the triple $\{\FF_{(1:0)},\FF_{(s_0:t_0)},\FF_{(0:1)}\}$ is a 3-web.
\end{proposition}

The uniqueness above follows from the fact that a Veronese curve in $\R P^1$ is determined by values at three distinct points. In higher dimensions there is no so simple correspondence between finite families of foliations and Veronese webs since one has to impose additional integrability conditions on function $w$. In particular in dimension 3 one gets the Hirota equation \cite{DK,Z}.

In what follow, for the sake of convenience, we will use the affine parameter $t=(1:t)\in\R P^1$ rather than the projective one $(s:t)$. The foliation corresponding to $(0:1)$ will be denoted $\FF_\infty$. Formula \eqref{eq2} can be equivalently written as
\begin{equation}\label{eq2b}
\omega_t=t_0w_xdx+tw_ydy
\end{equation}

We have investigated the geometry of Veronese webs in \cite{JK,K2}. In particular we have introduced a linear connection associated to a web. In the present paper we will denote it $\nn$. In dimension 2 it can be written explicitly in the following form
\begin{eqnarray}
\nn_{\partial_x}\partial_x=\frac{w_yw_{xx}-w_xw_{xy}}{w_xw_y}\partial_x, &\quad&
\nn_{\partial_y}\partial_y=\frac{w_xw_{yy}-w_yw_{xy}}{w_xw_y}\partial_y,\label{eqCon}\\
\nn_{\partial_x}\partial_y=0, &\quad&\nn_{\partial_y}\partial_x=0.\nonumber
\end{eqnarray}
On the other hand, for a 3-web there is a notion of the Chern connection (see \cite{N}). If a Veronese web is defined by a 3-web then $\nn$ coincides with the Chern connection of the 3-web. Indeed, it can be verified by a direct inspection that the formulae for $\nn$ can be computed as in \cite[Theorem 1.6]{N}.

\begin{proposition}\label{prop2}
Let $\{\FF_t\}$ be a Veronese web on $\R^2$. The associated connection $\nn$ has the following properties:
\begin{enumerate}
\item[(a)] All leaves of $\FF_t$ are geodesics of $\nn$ for any $t\in\R$.
\item[(b)] The torsion of $\nn$ vanishes.
\item[(c)] The Ricci curvature tensor of $\nn$ is skew-symmetric.
\end{enumerate}
\end{proposition}
\begin{proof}
We assume that a Veronese web is defined by \eqref{eq2b}. Then the leaves of $\FF_t$ are integral curves of the vector field
$$
w_y\partial_x-tw_x\partial_y.
$$
We directly compute
$$
\nn_{w_y\partial_x-tw_x\partial_y}(w_y\partial_x-tw_x\partial_y)= \left(\frac{w_y^2w_{xx}-twx^2w_{yy}}{w_xw_y}\right)(w_y\partial_x-tw_x\partial_y)
$$
and it proves Statement (a). Statement (b) immediately follows from the definition of $\nn$. To prove Statement~(c) we compute non-trivial components of the curvature (3,1)-tensor tensor $R(\nn)$. We get
$$
R(\nn)(\partial_x,\partial_y)\partial_x=\rho\partial_x,\qquad R(\nn)(\partial_x,\partial_y)\partial_y=\rho\partial_y,
$$
where
\begin{equation}\label{eqR}
\rho=\frac{w_{xx}w_{xy}}{w_x^2}-\frac{w_{yy}w_{xy}}{w_y^2}-\frac{w_{xxy}}{w_x}+\frac{w_{xyy}}{w_y}.
\end{equation}
It follows that the Ricci tensor of $\nn$ is represented by the matrix
$$
Ric(\nn)=\left(\begin{array}{cc} 0 & \rho \\ -\rho & 0 \end{array} \right).
$$
\end{proof}

It can be deduced from \cite[Corollary 7.5]{K2} that any torsion-free connection with skew-symmetric Ricci tensor can be obtained as a conneciton $\nn$ for a web. Indeed, in the proof of \cite[Corollary 7.5]{K2} with $m=1$ it is shown how to construct a Veronese web on a plane with a curvature being an arbitrary function. The result was obtained in terms of the canonical frames. Here we will show a different reasoning.

Let $\nabla$ be a torsion-free connection on a plane. If $Ric(\nabla)$ is skew-symmetric then it follows from linear algebra that there exists a function $\rho$ such that $R(\nabla)(\partial_x,\partial_y)V=\rho V$ for any vector field $V$. Let us fix a point $x\in\R^2$ and chose a frame $X(x),Y(x)\in T_x\R$. Moreover, for any other $y\in\R^2$ let us choose a smooth curve $\gamma_y$ joining $x$ and $y$ and define $X(y),Y(y)\in T_y\R^2$ by the parallel transport of $X(x)$ and $Y(x)$ along $\gamma_y$. In this way we construct two vector fields $X$ and $Y$. It follows from  the property $R(\nabla)(\partial_x,\partial_y)V=\rho V$ that the frame bundle of $T\R^2$ reduces to a $GL(1,\R)$-bundle and consequently the choice of different cures $\gamma_y$ leads to vector fields $\tilde X$ and $\tilde Y$ which are proportional to $X$ and $Y$ in the same way, i.e. $\tilde X=fX$ and $\tilde Y=fY$ for some function $f\colon\R^2\to\R$. We define a Veronese web $\{\FF_t\}$ imposing that the set of leaves of $\FF_t$ is the set of integral curves of the vector field
$$
X+tY.
$$
Note that $\tilde X+t\tilde Y$ has the same integral curves as $X+tY$ and hence the web is well defined. The choice of a different frame at the initial point $x$ leads to a M\"obius transformation of the projective parameter $(s:t)$ which parametrises the foliations. In this way we proved

\begin{theorem}\label{thm1}
There is a one-to-one correspondence between Veronese webs (given up to a M\"obius transformation of the projective parameter $(s:t)$) on a plane and torsion-free connections on $\R^2$ with skew-symmetric Ricci tensor.
\end{theorem}

We provide the following examples as applications of Theorem \ref{thm1}:
\vskip 1ex
1. {\bf Flat case.} It is clear that the flat connection corresponds to the linear function $w(x,y)=x+y$ and the associated web is defined by the one-form
$$
\omega_t=dx+tdy.
$$

\vskip 2ex
2. {\bf Constant curvature.} In order to find a torsion-free connection with constant skew-symmetric Ricci tensor one has to solve the equation
$$
\rho=C,
$$
where $\rho$ is given by \eqref{eqR} and $C\in\R$ is constant. The formula for $\rho$ can be written in more compact way
$$
\rho=\left(\frac{w_{xy}}{w_y}\right)_y -\left(\frac{w_{xy}}{w_x}\right)_x= \ln(w_y)_{xy}-\ln(w_x)_{xy}= \ln\left(\frac{w_y}{w_x}\right)_{xy}.
$$
Thus, the equation $\rho=C$ gives
$$
\frac{w_y}{w_x}=e^{Cxy}
$$
or $w_y=e^{Cxy}w_x$. This equation can be solved using the method of characteristics. However the knowledge of an exact solution is not necessary because we can always multiply the one-form $\omega_t$ from formula \eqref{eq2b} by a function and the resulting one-form defines the same Veronese web. Thus, multiplying $\omega_t$ by $w_x^{-1}$, we get that the web corresponding to a connection with $\rho=C$ is defined by the one-form
$$
dx+te^{Cxy}dy.
$$
The connection is given by
$$
\nn_{\partial_x}\partial_x=-Cy\partial_x,\quad\nn_{\partial_y}\partial_y=Cx\partial_y,
\quad \nn_{\partial_x}\partial_y=0,\quad \nn_{\partial_y}\partial_x=0.
$$
These formulae can be derived directly from equation \eqref{eqCon} because
$$
\frac{w_yw_{xx}-w_xw_{xy}}{w_xw_y}=-\frac{w_x}{w_y}\partial_x\left(\frac{w_y}{w_x}\right), \qquad \frac{w_xw_{yy}-w_yw_{xy}}{w_xw_y}=\frac{w_x}{w_y}\partial_y\left(\frac{w_y}{w_x}\right).
$$

\vskip 2ex
3. {\bf Wong's theorem.} Derdzinski \cite[Theorem 6.1]{Der} proved that for a torsion-free connection $\nabla$ with skew-symmetric Ricci tensor one can always choose local coordinates $(x_1,x_2)$ such that the Christoffel symbols have the form $\Gamma^1_{11}=-\partial_{x_1}f$, $\Gamma^2_{22}=\partial_{x_2}f$ for some function $f$ and $\Gamma^i_{jk}=0$ unless $i=j=k$.
In view of formula~\eqref{eqCon} and our Theorem \ref{thm1} this is evident as we can write $\nn_{\partial_x}\partial_x=-\partial_x\ln\left(\frac{w_y}{w_x}\right)\partial_x$ and $\nn_{\partial_y}\partial_y=\partial_y\ln\left(\frac{w_y}{w_x}\right)\partial_y$ if $\frac{w_y}{w_x}>0$ or $\nn_{\partial_x}\partial_x=-\partial_x\ln\left(-\frac{w_y}{w_x}\right)\partial_x$ and $\nn_{\partial_y}\partial_y=\partial_y\ln\left(-\frac{w_y}{w_x}\right)\partial_y$ if $\frac{w_y}{w_x}<0$. Note that the sign of $\frac{w_y}{w_x}$ is always fixed because both $w_y$ and $w_x$ never vanish since all foliations intersect transversely.

\section{Projective structures}
Two connections on a manifold $M$ are projectively equivalent if their sets of unparametrised geodesics coincide. A projective structure is a set of unparametrised geodesics of a connection, or, equivalently, it is a class of projectively equivalent connections.

Any projective structure on a plane can be locally described in terms of a second order ODE. Namely, fixing local coordinates $(x,y)$ one can look for an equation in the form
\begin{equation}\label{eq3}
y''=\Phi(x,y,y')
\end{equation}
such that the solutions $(x,y(x))$ are geodesics for the projective structure. It can be shown that the equation satisfies
\begin{equation}\label{eqcond}
\partial_{y'}^4\Phi=0
\end{equation}
and conversely any equation satisfying this condition defines a projective structure. The condition is point invariant and in fact any projective structure corresponds to a class of point equivalent equations.

We will show now that we can construct a projective structure out of a Veronese web. Indeed we have the following
\begin{proposition}\label{prop3}
If $\{\FF_t\}$ is a Veronese web on $\R^2$ then the union of all leaves of all foliations $\FF_t$ is a projective structure defined by the associated connection $\nn$. If $\{\FF_t\}$ is given by the one-form \eqref{eq2b} then the corresponding second order ODE is of the form
\begin{equation}\label{eq3b}
y''=\frac{1}{w_xw_y}\left((w_yw_{xx}-w_xw_{xy})y'+(w_yw_{xy}-w_xw_{yy})(y')^2\right).
\end{equation}
\end{proposition}
\begin{proof}
The first part follows directly from Statement (a) of Proposition \ref{prop2}. To get a description in terms of an ODE we recall that the leaves of $\FF_t$ are integral curves of the vector field $w_y\partial_x-tw_x\partial_y$. It follows that for a fixed $t$ they are solutions to the following first order equation
$$
y'=-t\frac{w_x}{w_y}.
$$
Differentiating this equation with respect to $x$ and eliminating parameter $t$ by substitution $t=-\frac{w_y}{w_x}y'$ we get \eqref{eq3b}.
\end{proof}

Moreover we have
\begin{lemma}\label{lemma1}
Equation \eqref{eq3b} is point equivalent to the derivative of a first order ODE.
\end{lemma}
\begin{proof}
Let $\phi\colon \R^2\to\R$ be a solution to
$$
\partial_y\phi=\frac{w_y}{w_x}.
$$
We define the following point transformation
$$
\tilde x=x,\qquad \tilde y=\phi(x,y)
$$
and verify that in the coordinates $(\tilde x, \tilde y)$ equation \eqref{eq3b} takes the form
\begin{equation}\label{eq3c}
\tilde y''=\phi_x(\tilde x,\phi^{-1}(\tilde x, \tilde y))_{\tilde x}+\phi_x(\tilde x,\phi^{-1}(\tilde x, \tilde y))_{\tilde y}\tilde y'.
\end{equation}
In above $\phi_x$ is the derivative of the mapping $(x,y)\mapsto \phi(x,y)$ with respect to the first coordinate, whereas $\phi_x(\tilde x,\phi^{-1}(\tilde x, \tilde y))_{\tilde x}$ and $\phi_x(\tilde x,\phi^{-1}(\tilde x, \tilde y))_{\tilde y}$ are derivatives of $(\tilde x,\tilde y)\mapsto \phi_x(\tilde x,\phi^{-1}(\tilde x, \tilde y))$ with respect to $\tilde x$ and $\tilde y$, respectively. The inverse $\phi^{-1}$ is taken with respect to the second coordinate function.

Equation \eqref{eq3c} is the derivative of
$$
\tilde y'=\phi_x(\tilde x,\phi^{-1}(\tilde x,\tilde y)).
$$
\end{proof}

As a corollary we get the following characterisation of linear connections projectively equivalent to a connection with skew-symmetric Ricci tensor.
\begin{theorem}\label{thm2}
Let $\nabla$ be a linear connection on $\R^2$. The following conditions are equivalent
\begin{enumerate}
\item[(a)] $\nabla$ is projectively equivalent to a connection with skew-symmetric Ricci curvature tensor.
\item[(b)] $\nabla$ is projectively equivalent to the Chern connection of a 3-web (or the connection $\nn$ associated to a Veronese web $\{\FF_t\}$).
\item[(c)] Unparametrised geodesics of $\nabla$ are described by solutions to a second order ODE which is the derivative of a first order ODE.
\end{enumerate}
\end{theorem}
\begin{proof}
The equivalence (a)$\iff$(b) follows from Theorem \ref{thm1} and the implication (b)$\implies$(c) follows from Lemma \ref{lemma1}. Therefore it is sufficient to prove that a second order ODE which is a derivative of a first order ODE gives a projective structure defined by a connection with skew-symmetric Ricci tensor. This fact was proved in \cite{DW} (see Theorem \ref{thmDW} below). Note that the condition \eqref{eqcond} is always satisfied for the derivatives of first order ODEs.
\end{proof}

\section{Twistor space and dual ODE}
The twistor space of a projective structure is the set of unparamterised geodesics. In the case of the projective structure on a plane the twistor space is a manifold of dimension two.

Theorem \ref{thm2} should be compared to the following result of Dunajski and West.
\begin{theorem}{\cite[Section 6.4, Proposition 3]{DW}}\label{thmDW}
There is a one-to-one correspondence between projective structures on a plane for which the twistor space fibers over $\R P^1$ and point equivalent classes of second order ODEs which are derivatives of first order ODEs.
\end{theorem}

The fibration over $\R P^1$ can be easily seen from the point of view of Veronese webs. Indeed, to a geodesic which is a leaf of $\FF_{(s:t)}$ one assigns the point $(s:t)\in\R P^1$.

In \cite{K2} we have characterised Veronese webs on in terms of ODEs in the following way (the analogous results are also proved for higher order ODEs and systems of ODEs).
\begin{theorem}{\cite[Theorem 1.1]{K2}}\label{thmK}
There is a one-to-one correspondence between Veronese webs on $\R^2$ and time-preserving contact equivalent classes of second order ODEs given in the form
\begin{equation}\label{eq4}
z''=F(t,z,z')
\end{equation}
for which the invariant
$$
K_0=-\partial_zF+\frac{1}{2}X_F(\partial_{z'}F)-\frac{1}{4}(\partial_{z'}F)^2,
$$
vanishes, where $X_F=\partial_t+z'\partial_z+F\partial_{z'}$ is the total derivative.
\end{theorem}

 The invariant $K_0$ is sometimes called the Jacobi endomorphism \cite{CMS} and it also appears in \cite{G} where is denoted $T$. It should be stressed that it is not a point invariant of the equation. It is invariant with respect to contact transformations which preserve the independent variable $t$ (see \cite{JK} for the general theory of such transformations). The class of transformations is strictly more rigid than the class of point transformations. Actually, one can also allow the M\"obius transformations of $t$ (i.e. transformations of the form $t\mapsto \frac{at+b}{ct+d}$ where $a,b,c,d\in\R$ are constant and satisfy $ad-bc\neq 0$) and $K_0$ remains invariant. The M\"obius transformations of the independent variable correspond to the transformations of the projective parameter which parametrises the corresponding Veronese web. Summarising, Theorem \ref{thmK} together with Theorem \ref{thm1} give the following characterisation of torsion-free connections with skew-symmetric Ricci tensor in terms of invariants of ODEs.
 
\begin{corollary}
There is a one-to-one correspondence between torsion-free connections with skew-symmetric Ricci tensor and second order ODEs satisfying $K_0=0$ and given modulo time-preserving contact transformations and M\"obius transformations of the independent variable.
\end{corollary}

Equation \eqref{eq4} is dual to equation \eqref{eq3} in the sense of the Cartan duality for second order ODEs. To be more precise, the class of point equivalent equations defined by \eqref{eq4} is dual to the class of point equivalent equations defined by \eqref{eq3}. Equation \eqref{eq4} is an equation on the twistor space, i.e. both $t$ and $z$ can be considered as coordinates on the twistor space. Additionally $t$ is exactly the parameter which defines the fibration over $\R P^1$. The Veronese web for equation \eqref{eq4} is defined on the space of its solutions which is the $(x,y)$-space for equation \eqref{eq3}. Conversely, the twistor space is the solutions space for equation \eqref{eq3}. Theorem \ref{thmK} together with Theorem \ref{thm2} give the following result, which, in a sense, is dual to the result of \cite{R}.

\begin{corollary}
A second order ODE is point equivalent to the derivative of a first order ODE if and only if its dual equation is point equivalent to an equation for which $K_0=0$.
\end{corollary}

\vskip 2ex
{\bf Acknowledgements.} I wish to thank Maciej Dunajski for useful discussions. The work has been partially supported by the Polish National Science Centre grant ST1/03902.


\begin{thebibliography}{99}
\bibitem{AG} M. A. Akivis, V. V. Goldberg, \textit{Differential geometry of webs}, In: Handbook of Differential Geometry, Vol. I, North-Holland, Amsterdam, 2000, 1-152, Chapter~1.
\bibitem{CMS} M. Crampin, E. Martinez, W. Sarlet, \textit{Linear connections for systems of second-order ordinary differential equations}, Ann. Inst. Henri Poincare 65, no. 2 (1996), p. 223-249.
\bibitem{Der} A. Derdzinski, \textit{Connections with skew-symmetric Ricci tensor on surfaces}, Results Math. 52 (2008), no. 3-4, 223-245.
\bibitem{Dun} M. Dunajski, \textit{A class of Einstein-Weyl spaces associated to an integrable system of 
hydrodynamic type}, J. Geom. Phys. 51 (2004), 126-137.
\bibitem{DK} M. Dunajski, W. Kry\'nski, \textit{Einstein--Weyl geometry, dispersionless Hirota equation and Veronese webs}, submitted (2013), arXiv:1301.0621. 
\bibitem{DW} M. Dunajski, S. West, \textit{Anti-Self-Dual Conformal Structures with Null Killing Vectors from Projective Structures}, Commun. Math. Phys. 272, 85-118 (2007).
\bibitem{GZ} I. M. Gelfand and I. Zakharevich, \textit{Webs, Veronese curves, and bi-Hamiltonian systems}, Journal of Functional Analysis, vol. 99, no. 1 (1991), pp. 150-178.
\bibitem{G} D. A. Grossman, \textit{Torsion-free path geometries and integrable second order ODE systems}, Selecta Mathematica, Volume 6, Issue 4, pp 399-342 (2000). 
\bibitem{JK} B. Jakubczyk, W. Kry\'nski, \textit{Vector fields with distributions and invariants of ODEs}, IMPAN preprint no. 728 (2010), latest version available at www.mimuw.edu.pl/$\sim$krynski/JGM2013.pdf, to appear (2013).
%\bibitem{K1} W. Kry\'nski, \textit{Paraconformal structures and differential equations}, Differential Geometry and its Applications, Volume 28, Issue 5, (2010), pp. 523-531.
\bibitem{K2} W. Kry\'nski, \textit{Geometry of isotypic Kronecker webs}, Central European Journal of Mathematics, Vol. 10, 1872-1888 (2012).
\bibitem{N} P. Nagy, \textit{Webs and curvature}, In: Web Theory and Related Topics, Toulouse, December, 1996, World Scientific Publishing, River Edge, 2001, 48-91.
\bibitem{R} M. Randall, \textit{Local obstructions to projective surfaces admitting skew-symmetric Ricci tensor}, arXiv:1302.4155 (2013).
\bibitem{W} Y.-C. Wang, \textit{Two dimensional linear connections with zero torsion and recurrent curvature}, Monatsh. Math. 68 (1964), 175-184. 
\bibitem{Z} I. Zakharevich, \textit{Nonlinear wave equation, nonlinear Riemann problem, and the twistor transform of Veronese webs}, arXiv:math-ph/0006001 (2000).
\end{thebibliography}
\end{document}